\documentclass[12pt,a4paper]{article}
\usepackage{amsmath,amssymb,amsthm,mathtools}
\usepackage[margin=1.15in]{geometry}
\usepackage{tgtermes}
\usepackage{xcolor}
\usepackage[colorlinks=true,linkcolor=green!35!black,citecolor=green!35!black,
            urlcolor=green!35!black]{hyperref}
\usepackage{setspace}

\newtheorem{theorem}{Theorem}[section]

\newtheorem{lemma}[theorem]{Lemma}

\theoremstyle{remark}\newtheorem{remark}[theorem]{Remark}

\DeclareMathOperator{\tr}{tr}
\DeclareMathOperator{\Cov}{Cov}
\newcommand{\R}{\mathbb{R}}\newcommand{\E}{\mathbb{E}}
\newcommand{\ip}[2]{\langle #1,#2\rangle}

\newcommand{\va}{\mathbf{a}}\newcommand{\vb}{\mathbf{b}}
\newcommand{\vu}{\mathbf{u}}\newcommand{\vx}{\mathbf{x}}
\newcommand{\vV}{\mathbf{V}}\newcommand{\vX}{\mathbf{X}}
\newcommand{\vY}{\mathbf{Y}}

\title{Two short proofs of incompatibility\\ for correlation matrices}
\author{Cedric Phillips\thanks{cedric.phillips@mail.polimi.it}}
\date{\vspace{-5ex}}

\begin{document}
\maketitle

\begin{abstract}
Let $P_d$ be the elliptope of $d\times d$ correlation matrices, and, for a
standardised law $F$, let $S_d^F$ be the set of correlation matrices of random
vectors with all margins $F$. We give two short proofs. First, an eleven-term
representation of $|\vx|^4$ as a sum of fourth powers of linear forms on $\R^4$ yields
$S_{11}\neq P_{11}$ for uniform margins. Second, the six diagonals of the regular icosahedron yield $S_6^F\neq P_6$ for
arcsine margins. Together with a result of Devroye and Letac, this gives
$S_d^F = P_d$ if and only if $d\le5$, settling a conjecture of theirs.
\end{abstract}

\section{Introduction}

For a random vector with continuous margins, Spearman's rho of a pair is the Pearson
correlation of the probability transforms. Write
\[
  P_d = \{\text{symmetric positive semidefinite }d\times d\text{ matrices with unit
  diagonal}\}
\]
for the elliptope, that is, the set of $d\times d$ correlation matrices, and $S_d$
for the set of matrices of pairwise Spearman coefficients. Embrechts, McNeil and
Straumann \cite{EMS2002} asked whether $S_d = P_d$. Devroye and Letac \cite{DL2015} proved equality for $d\le9$, and Wang, Wang and Wang \cite{WWW2019} proved $S_d\neq P_d$ for $d\ge12$ by exhibiting twelve
unit vectors in $\R^4$ satisfying an exact quartic identity.

While this work was in preparation, Wang and Zhang \cite{WZ2026} settled the
question completely, showing that $S_d = P_d$ if and only if $d\le9$ by constructing an explicit counterexample in dimension 10. Their work subsumes Theorem~\ref{thm:11} below, which we obtained independently, at about the same time (early August 2026). We exhibit our proof in dimension 11 because it is short, simple (using only a polynomial identity from the literature), and follows the idea set out by Wang,
Wang and Wang \cite{WWW2019} in their proof for $d \geq 12$. 

The second proof concerns a different margin. Its conclusion is new.
Call a law (probability distribution) $F$ on $\R$ standardised if it has mean $0$ and variance $1$, and
for such an $F$ write $S_d^F$ for the set of correlation matrices of $d$-dimensional
random vectors whose margins are all $F$. The Spearman matrix of a random
vector with continuous margins is the correlation matrix of its probability
transforms (rescaled to be standardised). Therefore, one has $S_d = S_d^u$ with $u$ the uniform
law on $[-\sqrt3,\sqrt3]$. In the same paper in which they proved $S_d = P_d$ for
$d\le9$, Devroye and Letac \cite[\S5]{DL2015} conjectured that the analogous
statement for the arcsine law (that is, the $\mathrm{Beta}(\frac{1}{2}, \frac{1}{2})$ distribution) fails when $d \geq 6$. 
Theorem~\ref{thm:arcsine} gives the exact threshold when combined with their result for $d \leq 5$.

\section{Reducing the problem}\label{sec:red}

Both proofs use the same idea, due to \cite{WWW2019} in the uniform case. As the following lemma only requires that the margins share one standardised law, we can work generally and apply it to both problems (we also give a different argument from the original). Recall that $A \in \R^{d\times k}$
of rank $k$ with $R = AA^{\!\top}$ is called a rank decomposition of
$R \in P_d$. The rows $\va_1^{\!\top},\dots,\va_d^{\!\top}$ of such an $A$ are unit vectors, since the $j$-th diagonal entry of $R$ is $|\va_j|^2$.

\begin{lemma}[{after \cite[Thm.~2.2]{WWW2019}}]\label{lem:www}
Let $F$ be standardised, let $R \in P_d$ have rank $k$, and let $A$ be a rank
decomposition of $R$. Then $R \in S_d^F$ if and only if there is a random vector $\vV$
in $\R^k$ with $\E \vV = 0$, $\Cov \vV = I_k$ and $\ip{\va_j}{\vV}\sim F$ for every
$j = 1,\dots,d$.
\end{lemma}

\begin{proof}
Suppose first that such a $\vV$ exists, and let $\vX = A\vV$, the random vector in $\R^d$
with coordinates $X_j = \ip{\va_j}{\vV}$. By assumption, each $X_j$ has law $F$, so $\vX$ has
all its margins equal to $F$; in particular, each $X_j$ has mean $0$ and variance $1$,
so the correlation matrix of $\vX$ is simply $\E[\vX\vX^{\!\top}]$. Since $\E \vV = 0$, we have
$\E[\vV\vV^{\!\top}] = \Cov \vV = I_k$ (the $k \times k$ identity matrix), and therefore
$\E[\vX\vX^{\!\top}] = A \, \E[\vV\vV^{\!\top}]A^{\!\top} = AA^{\!\top} = R$. Thus $R \in S_d^F$.

Conversely, suppose $R \in S_d^F$, and let $\vY$ be a random vector with margins $F$ and
correlation matrix $R$. Since $F$ is standardised, $\E \vY = 0$ and $\E[\vY\vY^{\!\top}] = R$.
The matrix $A$ has full column rank, so $A^{\!\top}A$ is invertible, and
$A^{+} = (A^{\!\top}A)^{-1}A^{\!\top}$ (the Moore--Penrose inverse of $A$) satisfies
$A^{+}A = I_k$. The matrix $P = AA^{+}$ is the orthogonal projection of $\R^d$ onto the
column space of $A$: it is symmetric, $P^2 = P$, and $PA = A$. Set $\vV = A^{+}\vY$. Then
$\E \vV = 0$ and
\[
  \Cov \vV\; = \; A^{+}R \, (A^{+})^{\!\top} \; = \; (A^{+}A)(A^{+}A)^{\!\top} \; = \; I_k .
\]
Moreover $\vY-A\vV = (I-P)\vY$, and $(I-P)^{\!\top}(I-P) = I-P$, so
\[
  \E\bigl|\vY-A\vV\bigr|^2 \; = \; \E\bigl[\vY^{\!\top}(I-P)\vY\bigr] \; = \; \tr\bigl((I-P)R\bigr)
   \; = \; \tr\bigl(((I-P)A)A^{\!\top}\bigr) \; = \; 0 ,
\]
because $(I-P)A = 0$. Hence $\vY = A\vV$ a.s., and $\ip{\va_j}{\vV} = Y_j$ has law $F$ for
every $j$.
\end{proof}

\begin{lemma}\label{lem:pad}
Let $F$ be standardised. If $M \in P_m\setminus S_m^F$ and $d\ge m$, then
$M\oplus I_{d-m} \in P_d\setminus S_d^F$.
\end{lemma}

\begin{proof}
The matrix $M\oplus I_{d-m}$ lies in $P_d$. If it were the correlation matrix of some
$\vX$ with margins $F$, then $(X_1,\dots,X_m)$ would have margins $F$ and correlation
matrix $M$, which would contradict $M\notin S_m^F$.
\end{proof}

\section{Uniform margins: dimension 11}\label{sec:11}

Let $u$ denote the uniform law on $[-\sqrt3,\sqrt3]$, which has mean $0$, variance
$1$ and $\E[u^4] = (\sqrt3)^4/5 = 9/5$. A random vector has Spearman matrix $R$ exactly
when it has margins $u$ and correlation matrix $R$, so $S_d = S_d^u$.

Equation (9.27)(i) of Reznick \cite{Reznick1992} reads
\begin{equation}\label{eq:reznick}
\begin{aligned}
  192 \, |\vx|^4 \; = \; &6 \, (x_1+x_2+x_3+x_4)^4+\Sigma_4 \, (3x_1-x_2-x_3-x_4)^4\\
  &+\Sigma_6\bigl((1{+}\sqrt2)x_1+(1{+}\sqrt2)x_2+(1{-}\sqrt2)x_3+(1{-}\sqrt2)x_4\bigr)^4,
\end{aligned}
\end{equation}
where $\Sigma_4$ is the sum of the four terms obtained by permuting the variables in
$3x_1-x_2-x_3-x_4$, and $\Sigma_6$ the sum of the six terms obtained by choosing
which two variables carry the coefficient $1+\sqrt2$. There are therefore $1+4+6 = 11$ terms.

\begin{theorem}\label{thm:11}
$S_d\neq P_d$ for every $d\ge11$.
\end{theorem}

\begin{proof}
Write $\vb_1,\dots,\vb_{11}$ for the coefficient vectors of the eleven linear forms in
\eqref{eq:reznick}, with multiplicities $c_1 = 6$ and $c_2 = \dots = c_{11} = 1$. Put
$\va_j = \vb_j/|\vb_j|$. The first vector has $|\vb_1|^2 = 4$ and the rest have
$|\vb_j|^2 = 12$, so \eqref{eq:reznick} reads
$192|\vx|^4 = \sum_jc_j|\vb_j|^4\ip{\va_j}{\vx}^4$, with
\[
  \sum_{j = 1}^{11}c_j|\vb_j|^4 \; = \; 6\cdot4^2+10\cdot12^2 \; = \; 1536 .
\]
Dividing by $1536$ gives us
\begin{equation}\label{eq:cub11}
  \sum_{j = 1}^{11}w_j \, \ip{\va_j}{\vx}^4 \; = \; \tfrac18|\vx|^4\qquad(\vx \in \R^4),
  \qquad w_1 = \tfrac1{16},\quad w_2 = \dots = w_{11} = \tfrac3{32},
\end{equation}
so that $\sum_jw_j = \tfrac1{16}+10\cdot\tfrac3{32} = 1$.

Let $A$ be the $11\times4$ matrix with rows $\va_1^{\!\top},\dots,\va_{11}^{\!\top}$, and
let $M = AA^{\!\top}$, the matrix with entries $M_{ij} = \ip{\va_i}{\va_j}$. Then $M$ is
symmetric and positive semidefinite with unit diagonal, so $M \in P_{11}$. The $\va_j$
must span $\R^4$ (otherwise they would all be orthogonal to some $\vb\neq 0$, and
\eqref{eq:reznick} at $\vx = \vb$ would give $0 = 192|\vb|^4$). Hence $A$ has rank $4$, so does
$M = AA^{\!\top}$, and $A$ is a rank decomposition of $M$.

Suppose $M \in S_{11}$, and let $\vV$ be as in Lemma~\ref{lem:www}, so that
$\ip{\va_j}{\vV}\sim u$ for every $j$. Substituting $\vx = \vV$ in \eqref{eq:cub11} and taking
expectations, using $\E[\ip{\va_j}{\vV}^4] = \E[u^4] = 9/5$,
\[
  \tfrac18 \, \E|\vV|^4 \; = \; \sum_{j = 1}^{11}w_j \, \E\bigl[\ip{\va_j}{\vV}^4\bigr]
   \; = \; \tfrac95\sum_{j = 1}^{11}w_j \; = \; \tfrac95 ,
  \qquad\text{so}\qquad \E|\vV|^4 = \tfrac{72}{5} = 14.4 .
\]
But $\E|\vV|^2 = \tr\Cov \vV = 4$, so Jensen's inequality gives
$\E|\vV|^4\ge(\E|\vV|^2)^2 = 16$, a contradiction. Hence
$M \in P_{11}\setminus S_{11}$. We extend this to all $d \ge 11$ using Lemma~\ref{lem:pad}.
\end{proof}

\begin{remark}\label{rem:sharp}
Eleven is the least number of terms for which an identity like
\eqref{eq:cub11} can hold (with nonnegative weights). Such an identity with $n$ terms is exactly a representation of $|\vx|^4$ as a
sum of $n$ fourth powers of real linear forms on $\R^4$. By
\cite[Prop.~9.26]{Reznick1992}, the least such $n$ is $\binom52+1 = 11$; the borderline
case $n = \binom52 = 10$ is excluded, per \cite[Prop.~9.2(ii)]{Reznick1992}. Such a representation would force the ten lines $\R \va_j$ to be equiangular, that is, to meet pairwise at one common angle, but $\R^4$ contains at most six equiangular lines \cite{LS1973}. Therefore, a new argument is needed for dimension 10, like the one Wang and Zhang \cite{WZ2026} provided.
\end{remark}

\section{Arcsine margins: dimension 6}\label{sec:6}

Let $F$ denote the standardised arcsine law, that is, the law of $\sqrt2\cos\Theta$
with $\Theta$ uniform on $[0,2\pi)$; equivalently the standardised
$\mathrm{Beta}(\tfrac12,\tfrac12)$ law on $[-\sqrt2,\sqrt2]$. It has mean $0$,
variance $2 \, \E[\cos^2\Theta] = 1$, and
\[
  \E[F^4] \; = \; 4 \, \E[\cos^4\Theta] \; = \; 4\cdot\tfrac38 \; = \; \tfrac32 .
\]

\begin{theorem}\label{thm:arcsine}
For the standardised arcsine law $F$,
\[
  S_d^F = P_d\qquad\Longleftrightarrow\qquad d\le5 .
\]
\end{theorem}

\begin{proof}
($\Leftarrow$) This is a result of Devroye and Letac \cite[\S4]{DL2015}: for every
$k\ge\tfrac12$ and every $d\le5$, each $d\times d$ correlation matrix is the correlation
matrix of a random vector whose margins are all $\mathrm{Beta}(k,k)$. The arcsine law is
the case $k = \tfrac12$. We can standardise using the map map $t\mapsto\sqrt2 \, (2t-1)$ and change no correlation (since it is affine). Then $P_d\subseteq S_d^F$ for $d\le5$. Since the opposite
inclusion $S_d^F\subseteq P_d$ holds for any $F$,  $S_d^F = P_d$.

($\Rightarrow$) We show that $S_6^F\neq P_6$; by Lemma~\ref{lem:pad} it follows that
$S_d^F\neq P_d$ for every $d\ge6$. Let
$ \varphi = \tfrac{1+\sqrt5}{2}$, so that $ \varphi^2 = \varphi+1$. Consider the six
vectors
\[
\begin{aligned}
  \vu_1 & = (0,1, \varphi), & \vu_3 & = (\varphi,0,1), & \vu_5 & = (1, \varphi,0),\\
  \vu_2 & = (0,-1, \varphi), & \vu_4 & = (\varphi,0,-1), & \vu_6 & = (-1, \varphi,0),
\end{aligned}
\]
that is, one vertex from each pair of opposite vertices of a regular icosahedron. Using
$(p+q)^4+(p-q)^4 = 2 \, (p^4+6p^2q^2+q^4)$ on the pairs $\{\vu_1,\vu_2\}$,
$\{\vu_3,\vu_4\}$ and $\{\vu_5,\vu_6\}$,
\[
\begin{aligned}
  \ip{\vu_1}{\vx}^4+\ip{\vu_2}{\vx}^4& = 2 \, (\varphi^4x_3^4+6 \varphi^2 x_2^2 x_3^2+x_2^4),\\
  \ip{\vu_3}{\vx}^4+\ip{\vu_4}{\vx}^4& = 2 \, (\varphi^4x_1^4+6 \varphi^2 x_1^2 x_3^2 +x_3^4),\\
  \ip{\vu_5}{\vx}^4+\ip{\vu_6}{\vx}^4& = 2 \, (\varphi^4x_2^4+6 \varphi^2 x_1^2 x_2^2 +x_1^4).
\end{aligned}
\]
Summing up, the coefficient of each $x_i^4$ is $2(1 +  \varphi^4)$, and that of each
$x_i^2 x_j^2 $ is $12 \varphi^2 $. Since $\varphi^4 = (\varphi + 1)^2 = 3 \varphi+2$, we find the sum to be $6 \varphi^2 \bigl(\sum_ix_i^4 + 2\sum_{i<j}
x_i^2 x_j^2 \bigr) = 6 \varphi^2 |\vx|^4$. Each $\vu_j$ has $|\vu_j|^2 = 1 + \varphi^2 = \varphi+2$, and
$(\varphi+2)^2 = 5 \varphi^2 $. Therefore, the unit vectors $\va_j = \vu_j/|\vu_j|$
satisfy
\begin{equation}\label{eq:icos}
  \sum_{j = 1}^{6}\ip{\va_j}{\vx}^4 \; = \; \frac{6 \varphi^2 }{(\varphi + 2)^2 } \, |\vx|^4
   \; = \; \tfrac65 \, |\vx|^4\qquad(\vx \in \R^3).
\end{equation}

Let $A$ be the $6\times3$ matrix with rows $\va_1^{\!\top},\dots,\va_6^{\!\top}$, and let
$M = AA^{\!\top}$, the matrix with entries $M_{ij} = \ip{\va_i}{\va_j}$. Then $M$ is symmetric, positive semidefinite, and has unit diagonal, so $M \in P_6$. The $\va_j$ must span $\R^3$ (otherwise they would all be orthogonal to some $\vb\neq 0$, and \eqref{eq:icos} at $\vx = \vb$
would give $0 = \tfrac65|\vb|^4$). So $A$ has rank $3$, as does $M = AA^{\!\top}$, meaning that $A$
is a rank decomposition of $M$.

Suppose now that $M \in S_6^F$. Let $\vV$ be as in Lemma~\ref{lem:www}, so that
$\ip{\va_j}{\vV}\sim F$ for every $j$. Substituting $\vx = \vV$ in \eqref{eq:icos} and taking
expectations, using $\E[\ip{\va_j}{\vV}^4] = \E[F^4] = \tfrac32$,
\[
  \tfrac65 \, \E|\vV|^4 \; = \; \sum_{j = 1}^{6}\E\bigl[\ip{\va_j}{\vV}^4\bigr] \; = \; 6\cdot\tfrac32 \; = \; 9,
  \qquad\text{so}\qquad \E|\vV|^4 = \tfrac{15}{2} = 7.5 .
\]
But $\E|\vV|^2 = \tr\Cov \vV = 3$, so Jensen's inequality gives $\E|\vV|^4\ge(\E|\vV|^2)^2 = 9$, a
contradiction. Hence $M \in P_6 \setminus S_6^F$.
\end{proof}

This settles a conjecture of Devroye and Letac, who write that they ``conjecture
the existence of $R \in \mathcal R_6$ which cannot be the correlation of a
distribution whose margins are the arcsine distribution'' \cite[\S5]{DL2015} (their
$\mathcal R_6$ is our $P_6$).

\subsection*{Acknowledgements and disclosures}
Theorem~\ref{thm:11} was obtained independently of, and at approximately the same
time as, the work of Wang and Zhang \cite{WZ2026}, which settles the uniform case
completely. We are grateful for the clarity of their
construction in dimension 10, which cannot be obtained by the route taken here for dimension 11.

As in \cite{WZ2026}, a Large Language Model (LLM) was used in the production of this work. Specifically, the author suggested to the LLM that the idea of using equally-weighted vectors, as \cite{WWW2019} did, was an unnecessary restriction. The LLM was then able to pull Reznick's identity from the literature to provide a counterexample. The LLM also generated the new argument used to prove Lemma~\ref{lem:www}, and the configuration used for the forward implication in the proof of Theorem~\ref{thm:arcsine}. After consultation with the LLM, the work was written by a human author, who takes full responsibility for the contents.

\end{document}